\documentclass[a4paper,12pt]{article}

\usepackage{relsize,exscale}
\usepackage{color}
\usepackage{tikz,epic,eso-pic}

\usepackage{hyperref}
\usepackage[arrow,matrix]{xy}
\usepackage{makeidx}
\usepackage{bm}
\usepackage{latexsym}
\usepackage{amscd}
\usepackage{amsmath}
\usepackage{amssymb}
\usepackage{hhline}

\usepackage{stmaryrd}

\usepackage{amsthm}
\usepackage{float}
\usepackage{psfrag}
\usetikzlibrary{intersections}
\colorlet{examplefill}{yellow!80!black}

\theoremstyle{plain}
\newtheorem{theorem}{Theorem}[section]

\newtheorem{lemma}[theorem]{Lemma}

\newtheorem{proposition}[theorem]{Proposition}

\newtheorem{question}[theorem]{Question}

\theoremstyle{definition}
\newtheorem{definition}[theorem]{Definition}
\newtheorem{remark}[theorem]{Remark}

\newcommand{\Z}{\mathbb{Z}}

\newcommand{\R}{\mathbb{R}}
\newcommand{\C}{\mathbb{C}}
\newcommand{\K}{\mathbb{K}}
\newcommand{\PP}{\mathbb{P}}

\newcommand{\ch}{\operatorname{ch}}
\newcommand{\bch}{\operatorname{bch}}
\newcommand{\uch}{\operatorname{uch}}
\newcommand{\Sep}{\operatorname{Sep}}

\newcommand{\A}{{\mathcal A}}
\newcommand{\B}{{\mathcal B}}

\newcommand{\FF}{{\mathcal F}}
\newcommand{\LL}{{\mathcal L}}

\newcommand{\barH}{\overline{H}}
\newcommand{\barA}{\overline{\A}}
\newcommand{\barC}{\overline{C}}

\newcommand{\barF}{\overline{F}}
\newcommand{\barX}{\overline{X}}

\definecolor{deepblue}{cmyk}{0,0.83,1,0.70}
\definecolor{gray}{cmyk}{0,0,0,0.3}
\definecolor{rred}{cmyk}{0,1,1,0}
\definecolor{chairo}{cmyk}{0,0.83,1,0.70}
\definecolor{roypur}{cmyk}{0.75,0.90,0,0.1}
\definecolor{darkorc}{cmyk}{0.40,0.80,0.20,0}
\definecolor{oliv}{cmyk}{0.64,0.00,0.75,0.56}
\definecolor{azuro}{cmyk}{1,1,0,0.46}

\DeclareMathOperator{\Hom}{Hom}
\title{Maximal Betti number for local system cohomology of hyperplane arrangement complements}
\author{Yongqiang Liu\thanks{Institute of Geometry and Physics, University of Science and Technology of China, No 96 Jinzhai Road,  Hefei 230026, China
Email: liuyq@ustc.edu.cn} and 
Masahiko Yoshinaga\thanks{ 
Department of Mathematics, 
Graduate School of Science, Osaka University,
Toyonaka, Osaka 560-0043, Japan
E-mail: yoshinaga@math.sci.osaka-u.ac.jp}}
\date{\today}
\begin{document}
\maketitle

\begin{abstract}
Let $\LL$ be a rank one local system with field coefficient on the complement $M(\A)$ of an essential complex hyperplane arrangement $\A$ in $\C^\ell$. Dimca-Papadima and Randell independently showed that  $M(\A)$ is homotopy equivalent to a minimal CW-complex. It implies that $\dim H^k(M(\A),\LL) 
\leq b_k(M(\A))$. %When $\K=\C$, this inequality is first proved by Aomoto and Kita and later in full generality by Cohen. 
In this paper, we show that if $\A$ is real, then the inequality holds as equality  for some $0\leq k\leq \ell$ if and only if $\LL$ is the constant sheaf. The proof is using the descriptions of local system cohomology of $M(\A)$ in terms of chambers. 
%We also prove a similar result for the Aomoto complex.

\noindent
{\bf MSC-class}: 32S22 (Primary) 52C35, 32S50 (Secondary)
\end{abstract}

%\tableofcontents

\section{Introduction}

Let $\A=\{H_1, \dots, H_n\}$ be an essential hyperplane arrangement in 
$\C^\ell$, $M(\A)=\C^\ell\smallsetminus\bigcup_{H\in\A}H$ be its complement. 
%We also fix a defining equation $\alpha_i$ of $H_i$. 
%An arrangement $\A$ is called essential if normal vectors of hyperplanes  generate $\C^\ell$. 
The first homology group $H_1(M(\A), \Z)$ is  a free abelian group generated by the meridians $\{\gamma_1, \dots, \gamma_n\}$  of hyperplanes. 
%We denote their dual basis by $e_1, \dots, e_n\in H^1(M(\A), \Z)$. The element $e_i$ can be identified with  $\frac{1}{2\pi\sqrt{-1}}d\log\alpha_i$ via the de Rham isomorphism. 
Fix a field $\K$. 
A rank one $\K$-local system $\LL$ on $M(\A)$ is 
determined by a homomorphism $\rho:H_1(M(\A), \Z)\longrightarrow
\K^\times$, which is determined by an $n$-tuple 
$q=(q_1, \dots, q_n)\in(\K^\times)^n$, where 
$q_j=\rho(\gamma_j)$. 

The hyperplane arrangement complement $M(\A)$ is homotopy equivalent to a minimal CW-complex, proved by Dimca-Papadima \cite{dim-pap} and Randell \cite{ran},  independently. 
Hence for any integer $k\geq 0$, $H^k(M(\A),\Z)$ is free abelian and 
\begin{equation} \label{inequality}
    \dim H^k(M(\A),\LL) \leq b_k(M(\A)).
\end{equation}
When $\K=\C$, this inequality  is first proved for almost all choice of $ q$ by Aomoto-Kita \cite[Proposition 2.1]{ao-ki}  and later in full generality by Cohen \cite{Cohen98}.

One may wonder when does this inequality holds as equality. In this paper, we give a complete answer to this question when $\A$ is real. Note that since $\A$ is essential, $b_k(M(\A))>0$ exactly for $0\leq k \leq \ell$. 
\begin{theorem}
\label{thm main 1} 
Let $\A=\{H_1, \dots, H_n\}$ be an essential affine hyperplane arrangement 
in $\R^\ell$. If %$\LL$ is a rank one local system with $\K$-coefficients on $M(\A)$ such that 
(\ref{inequality}) holds as equality for some $0\leq k\leq  \ell$, then $\LL$ is the constant sheaf.
\end{theorem}

Our proof for Theorem \ref{thm main 1} relies on several works (\cite{y-lef, y-cham, y-mini}) 
concerning minimality of arrangements due to the second named author and the  deletion-restriction method for certain non-trivial rank one local systems discovered by Cohen \cite[Theorem 4]{Cohen02}.
Theorem \ref{thm main 1} leads to the following two questions.

\begin{question} \label{que 1} Does Theorem \ref{thm main 1} hold for complex hyperplane arrangement?
\end{question}
\begin{question} \label{que 2} Let $\LL$ be a rank $r$ local system with $\K$-coefficients on $M(\A)$. Due to the minimality of $M(\A)$, we have 
 $$ \dim H^k(M(\A),\LL) \leq r \cdot b_k(M(\A)) .$$ 
 If this inequality holds as equality for some   $0\leq k\leq  \ell$, does $\LL$ have to be the constant sheaf?
\end{question}
Clearly the answers are yes for both questions when $k=0$, then the same answers hold when $k=1$. So both questions are only open for $k\geq 2$.  One may reduce the study of Question \ref{que 2} to simple local systems.

Let $X$ be a minimal CW-complex homotopy equivalent to $M(\A)$. 
Note that  $H_1(X,\Z) \cong H_1(M(\A), \Z)\cong \Z^n$. Denote $\tilde{X}$ the maximal abelian covering of $X$. The cell structure of $X$ lifts to cell structures on  $\tilde{X}$ with the actions of  $\Z^n$ by deck transformations. 
The cellular chain complex of $\tilde{X}$ with $\K$-coefficients is a bounded complex of finitely generated free $\K[\Z^n]$-modules.
Consider its dual complex $C^{*}(\tilde{X}; \K)$:
\begin{equation}\label{cochain}
\cdots  \to C^{k-1}(\tilde{X}, \K) \overset{\partial_{k-1}}{\to} C^k(\tilde{X}, \K) \overset{\partial_{k}}{\to} C^{k+1}(\tilde{X}, \K)  \to   \cdots,
\end{equation} 
where $ C^{k}(\tilde{X}, \K)=\Hom(C_{k}(\tilde{X}, \K),\K[\Z^n])$ as $\K[\Z^n]$-modules. Note that $\K[\Z^n]\cong \K[t_1^{\pm 1},\cdots, t_n^{\pm 1}]$. Then the following complex of $\K$-vector spaces $$C^{*}(\tilde{X}; \K)\otimes_{\K[\Z^n]} \K[t_1^{\pm 1},\cdots, t_n^{\pm 1}]/(t_1-q_1,\cdots,t_n-q_n)$$  computes $H^*(M(\A),\LL)$. We will show in Remark \ref{rem nonzero} that if  $\LL$ is not the constant sheaf and $\A$ is real,  $\partial_k(q) \neq 0$ for all $0\leq k\leq \ell -1$. Since $ C^{k}(\tilde{X}, \K)$ is a free  $\K[\Z^n]$-module with rank being $b_k(M(\A))$, this implies that in this case, 
\begin{center}
    $ \dim H^k(M(\A),\LL)\leq b_k(M(\A))-2$ for any $1\leq k \leq \ell-1$.
\end{center}

On the other hand, denote the dual basis of $H_1(M(\A),\Z)\cong \Z^n$ by \{$e_1, \dots, e_n\}$ with $e_i \in H^1(M(\A), \Z)$. 
 Consider  $w=\sum_{i=1}^n w_i e_i \in H^1(M(\A),\K)$ with $w_i \in \K$ for all $i$. Note that $w\cup w=0,$ hence  we get the following Aomoto complex by cup product $(H^{\ast}(M(\A),\K), w\cup)$:
$$\cdots \to  H^{k-1}(M(\A),\K)\overset{ w\cup}{\to} H^{k}(M(\A),\K) \overset{ w\cup}{\to} H^{k+1}(M(\A),\K)\to\cdots
$$
Papadima-Suciu showed in \cite[Theorem 12.6]{ps} that the linearization of the
cochain complex $C^{*}(\tilde{X}; \K)$ coincides with the so-called universal Aomoto complex. See loc. cit. for more details. 
Then the following proposition gives some evidence to Question \ref{que 1}.
\begin{proposition} \label{prop cup product} Let $\A$ be an essential hyperplane arrangement in $\C^\ell$. 
    For any non-zero $w\in H^1(M(\A),\K)$,  the cup product map $  H^{k-1}(M(\A),\K)\overset{ w\cup}{\to} H^{k}(M(\A),\K)$ is non-zero for any $1\leq k\leq \ell$.
\end{proposition}

This paper is organized as follows. In Section \ref{sec:notation}, we review basic notations of real arrangements
and the methods to calculate local system cohomology in terms of chamber basis. 
Then we prove Theorem \ref{thm main 1} in Section \ref{sec:thm} and Proposition \ref{prop cup product} in Section \ref{sec:prop}.

\section{Notations and Preliminaries for real hyperplane arrangements} 

\label{sec:notation}

\subsection{Chambers on real arrangements}

\label{subsec:chambers}
Let $\A=\{H_1,\hdots,H_n\}$ be an affine 
hyperplane arrangement in $\R^\ell$. 
Denote by $M(\A)=\C^\ell \smallsetminus \cup_{i=1}^n H_i \otimes \C$ 
the complement of the complexified hyperplanes.
By identifying $\R^\ell$ with $\PP_{\R}^\ell\smallsetminus\barH_\infty$, 
define the projective closure by 
$\barA=\{\barH_1,\hdots,\barH_n,\barH_\infty\}$, where 
$\barH_i\subset\PP_{\R}^\ell$ is the closure of $H_i$ in the projective space. 
Let $L(\A)$ and $L(\barA)$ denote the intersection posets of 
$\A$ and $\barA$, respectively, namely, the poset of subspaces 
obtained as intersections of some hyperplanes with reverse inclusion order. 
An element of $L(\A)$ (and $L(\barA)$) is  called an edge. 
Denote by $L_k(\A)$ the set of all $k$-dimensional edges. 
For example $L_{\ell}(\A)=\{\R^\ell\}$ and $L_{\ell-1}(\A)=\A$. 
Then $\A$ is essential if and only if $L_0(\A)\neq\emptyset$.

Next we recall the description of the minimal complex 
in terms of real structures from \cite{y-lef, y-mini, y-cham}. 
Let $\A=\{H_1, \dots, H_n\}$ be an essential hyperplane arrangement in $\R^\ell$. 
A connected component of $\R^\ell\smallsetminus\bigcup_{i=1}^n H_i$ is called a 
chamber, and denote by $\ch(\A)$ the set of all chambers. 
Let $\bch(\A)$ (resp. $\uch(\A)$) denote the set of all bounded (resp. unbounded) chambers.
%A chamber $C\in\ch(\A)$ is called a bounded chamber if $C$ is bounded. The set of all bounded chambers of $\A$ is denoted by $\bch(\A)$. 
For a chamber $C\in\ch(\A)$, 
denote by $\barC$ the closure of $C$ in $\PP_{\R}^\ell$.
Then 
a chamber $C$ is bounded if and only if $\barC\cap\barH_\infty=\emptyset$. 
\begin{definition}
    For given two chambers $C, C'\in\ch(\A)$, the set  $\Sep(C, C')$ of  separating hyperplanes of $C$ and $C'$ is defined as
\[
\Sep(C, C'):=\{H_i\in\A\mid \mbox{ $H_i$ separates $C$ and $C'$}\}.
\]
\end{definition}

Let $\FF$ be a generic flag in $\R^\ell$, that is, $\FF$ is a sequence of affine subspaces:
\[
\FF: \emptyset=F^{-1}\subset F^0\subset F^1\subset\dots\subset F^{\ell}
=\R^\ell,
\]
with %$F^k$ is a $k$-dimensional affine subspace such that 
$\dim(\barX\cap\barF^k)=\dim \barX+k-\ell$ for any 
$\barX\in L(\barA)$. 
%The genericity of $\FF$ is equivalent to 
%\[
%F^k\cap L_i(\A)=L_{k+i-\ell}(\A\cap F^k), 
%\]
%for any $k+i\geq\ell$. 
From this point, we always assume that the generic flag $\FF$ is 
near to $\barH_\infty$, that is, $F^{k-1}$ does not separate $L_0(\A\cap F^{k})$ for all 
$k=1, \dots, \ell$.  
Denote
\[
\begin{split}
\ch^k(\A)
&=
\{C\in\ch(\A)\mid C\cap F^k\neq\emptyset, C\cap F^{k-1}=\emptyset\}\\
\bch^k(\A)
&=
\{C\in\ch^k(\A)\mid C\cap F^k\mbox{ is bounded}\}\\
\uch^k(\A)
&=
\{C\in\ch^k(\A)\mid C\cap F^k\mbox{ is unbounded}\}.\\
\end{split}
\]
Then clearly, we have 
$\ch^k(\A)=\bch^k(\A)\sqcup\uch^k(\A)$ and $ 
\ch(\A)=\bigsqcup_{k=0}^\ell\ch^k(\A).$
Since we assumed that $\FF$ is 
near to $\barH_\infty$, we have $\bch^\ell(\A)=\bch(\A)$. However, for any $0\leq k<\ell$, 
$C\in\bch^k(\A)$ is an unbounded chamber. 
%We recall the following proposition about  $\#\ch^k(\A)$.

\begin{proposition}
\label{prop:bchuch}
(\cite{y-lef, y-mini}) With the above notations, we have
\begin{itemize}
\item[$(i)$] $\#\ch^k(\A)=b_k(M(\A))$. 
\item[$(ii)$] $\#\bch^k(\A)=\#\uch^{k+1}(\A)$. 
%\item[$(iii)$] $\#\bch^k(\A)=b_k-b_{k-1}+\dots+(-1)^kb_0$. 
\end{itemize}
\end{proposition}
Concerning $(ii)$ of Proposition \ref{prop:bchuch}, there is an operation called taking the opposite chamber as follows.

\begin{definition}(\cite[Definition 2.1]{y-mini})
Let $C\in \uch(\A).$ 
There exists a unique chamber, denoted by $C^{\vee}\in\uch(\A)$, 
which is the opposite with respect to $\overline{C}\cap \barH_\infty,$ 
where $\overline{C}$ is the closure of $C$ in the projective space 
$\PP_{\R}^\ell$. 
\end{definition}
Since $(C^\vee)^\vee=C$, we obtain a bijection \[
\iota: \bch^k(\A)\stackrel{\simeq}{\longrightarrow}\uch^{k+1}(\A), 
C\longmapsto C^\lor, 
\]
which implies Proposition \ref{prop:bchuch}$(ii)$. See Figure \ref{fig:opposite} for an example.
\begin{figure}[htbp]
\centering
\begin{tikzpicture}[scale=1]

%�i�q
%\draw [help lines] (0,0) grid (9,5);%(0,0)����(10,4)�܂ł�"�א��̕���"

\draw[thick,rounded corners=0.6cm] (1,1) -- (8,1) -- (8,4.5) node[right] {$\overline{H}_\infty$};

\draw[thick] (1,2.5) node [left] {$H_1$} --(8.5,3.5);
\draw[thick] (1,3.5) node [left] {$H_2$} --(8.5,1.5);

\draw[thick,rounded corners=0.3cm] (5,0.5) -- (3,1.5) -- (3,4.5) node [above] {$H_3$}; 
\draw[thick,rounded corners=0.3cm] (4,0) -- (4,4.5) node [above] {$H_4$}; 
\draw[thick,rounded corners=0.3cm] (3,0.5) -- (5,1.5) -- (5,4.5) node [above] {$H_5$}; 

\filldraw[fill=black, draw=black] (4,1) circle (2pt);

\draw (2,1.5) node[above] {$C_1$};
\draw (3.5,1.5) node[above] {$C_2$};
\draw (4.5,1.5) node[above] {$C_3$};
\draw (6,1.5) node[above] {$C_4$};

\draw (2,3.5) node[above] {$C_4^\lor$};
\draw (3.5,3.5) node[above] {$C_2^\lor$};
\draw (4.5,3.5) node[above] {$C_3^\lor$};
\draw (6,3.5) node[above] {$C_1^\lor$};

\draw (2,0) node[above] {$C_1^\lor$};
\draw (3.5,0) node[above] {$C_3^\lor$};
\draw (4.5,0) node[above] {$C_2^\lor$};
\draw (6,0) node[above] {$C_4^\lor$};

\draw[very thick] (1,1) node[above] {\footnotesize $\overline{C_1}\cap\overline{H}_\infty$} --(4,1);

\end{tikzpicture}
\caption{Opposite chambers}
\label{fig:opposite}
\end{figure}

%Note that $C$ and $C^\lor$ are lying in the same side with respact to 
%$H\in\A$ if and only if $\barH\supset X(C)$. 
%Hence we have 

\begin{proposition} 
\label{prop:charsep} For $C\in \uch(\A),$ let us denote the projective subspace generated by $\barC\cap\barH_\infty$ 
by $X(C)=\langle\barC\cap\barH_\infty\rangle$. Then
we have
\begin{equation}
\label{eq:charsep}
\Sep(C, C^\lor)=\{H\in\A\mid \barH\not\supset X(C)\}=
\barA\smallsetminus\barA_{X(C)}. 
\end{equation}
In particular, if $\dim X(C)=\ell-1$, then $\Sep(C, C^\lor)=\A$. 
\end{proposition}

Next we define the degree map 
\[
\deg:\ch^k(\A)\times\ch^{k+1}(\A)\longrightarrow\Z. 
\]
Let $B=B^k\subset F^k$ be a 
$k$-dimensional ball with sufficiently large radius so that 
every $0$-dimensional edge $X\in L_0(\A\cap F^k)\simeq L_{\ell-k}(\A)$ 
is contained in the interior of $B^k$. 
Let $C\in\ch^k(\A)$ and $C'\in\ch^{k+1}(\A)$. 
Then there exists a vector field $U^{C'}$ on $F^k$ 
(\cite{y-lef}) which satisfies the following conditions. 
\begin{itemize}
\item 
$U^{C'}(x)\neq 0$ for $x\in\partial (\barC\cap B)$. 
\item 
Let $x\in\partial B$. Then 
$U^{C'}(x)$ directs to the inside of $B$. 
\item 
If $x\in H$ for some $H\in\A$, then 
$U^{C'}(x)\not\in T_x H$ and  directs the side 
containing $C'$. 
\end{itemize}

\begin{definition}
\label{degreemap}
For any $C\in\ch^k(\A)$ and 
$C'\in\ch^{k+1}(\A)$,   the degree of $C$ and $C'$ is defined by 
\[
\deg(C,C'):=\deg\left(\left.
\frac{U^{C'}}{|U^{C'}|}\right|_{\partial (\barC\cap B^k)}:
\partial(\barC\cap B^k)\longrightarrow S^{k-1}\right)\in\Z. 
\]
This is independent of the choice of $U^{C'}$ (\cite{y-lef}). 
\end{definition}

It is difficult to calculate degrees in general. However, for the opposite chambers, one can compute it as follows due to  Bailet and the second named author. 
\begin{proposition}\cite[Theorem 4.8]{BY}
    \label{prop:degree}
Let $C\in\bch^{\ell-1}(\A)$. Then 
\begin{equation}
\deg(C, C^\lor)=(-1)^{\ell-1-\dim X(C)}. 
\end{equation}
\end{proposition}

\subsection{Local system cohomology}

A rank one $\K$-local system $\LL$ on $M(\A)$ is 
determined by the monodromy $q_i\in\K^\times$ around each hyperplane 
$\barH_i$. We denote $q_\infty=(q_1q_2\cdots q_n)^{-1}$, which is the monodromy around the hyperplane $\barH_\infty$. 
%and $q_X=\prod_{\barH_i\supset X}q_i$ for an edge $X\in L(\barA)$. 
Since $\dim H^*(M(\A),\LL)$ does not change  with respect to field extension, without loss of generality we always assume that $\K$ is algebraically closed.
Fix $q_i^{1/2}=\sqrt{q_i}$. Set $q_{\infty}^{1/2}:=\left(q_1^{1/2}\cdots q_n^{1/2}\right)^{-1}$ and 
\[
\Delta(C, C'):=
\prod_{H_i\in\Sep(C, C')}q_i^{1/2}-
\prod_{H_i\in\Sep(C, C')}q_i^{-1/2}. 
\]
 Define the linear map 
$\nabla_{\LL}:\K[\ch^k(\A)]\longrightarrow\K[\ch^{k+1}(\A)]$ by 
\[
\nabla_{\LL}([C])= 
\sum_{C'\in \ch^{k+1}} \deg(C,C')\cdot 
\Delta(C, C')\cdot [C'].
\]
Then we have the following result due to the second named author. 

\begin{theorem}
  \label{thm:localchamber}
\cite{y-lef}
$(\K[\ch^\bullet(\A)],\nabla_{\LL})$ is a cochain complex. 
Furthermore, there is a natural isomorphism of cohomology groups for any $k\geq 0$: 
\[
H^k(\K[\ch^\bullet(\A)],\nabla_{\LL}) \simeq H^k(M(\A), \LL).
\]  
\end{theorem}

%\begin{remark}\label{rk1}
%We can check easily that $\lambda_\infty\in\Z \Leftrightarrow q_\infty=1$ and $\lambda_X\in\Z \Leftrightarrow q_X=1,\,$ for all $X\in L(\bar{\A}).$
%\end{remark}

\section{Proof of Theorem \ref{thm main 1}} 

\label{sec:thm}

%\begin{lemma}\label{lem 1}
 %  Let $\LL$ be a rank one local system with $\K$-coefficients  on $M(\A)$ such that $\dim H^k(M(\A),\LL) =b_k(M(\A)) $ for some $0\leq k\leq \ell $.
  % Consider $\A \cap F^{k}$. Let $\LL^k$ denote the induced local system on $M(\A \cap F^k)$. 
  % Then we have $\dim H^k(M(\A\cap F^{k}),\LL^k) =b_k(M(\A\cap F^k)) .$ 
%\end{lemma}
%\begin{proof}
%    By Lefschetz hyperplane section theorem, we have an injective map 
%    $$H^k(M(\A),\LL) \hookrightarrow H^k(M(\A\cap F^{k}),\LL^k) .$$
%    Then we have
%    $$  b_k(M(\A))=\dim H^k(M(\A),\LL) \leq \dim H^k(M(\A\cap F^k),\LL^k) \leq b_k(M(\A\cap F^k)).$$
%    Due to the genericity of $F^k$, we have $ b_k(M(\A))= b_k(M(\A\cap F^k))$, see e.g.  .
%     The claim follows
%\end{proof}

 First we recall the deletion-restriction method for certain non-trivial rank one local systems discovered by Cohen \cite[Theorem 4]{Cohen02}. The choice of a distinguished  hyperplane $\barH$ in $ \barA$ gives rise to a triple of arrangements $(\A, \A' , \A'' ),$ where $\A' = \barA\setminus \{\barH\}$
and $\A''$ is the arrangement in $\barH$ obtained by $\A'\cap \barH$. 
Let $M = M(\A)$, $M' = M(\A' )$, and $M'' = M(\A'' )$ be the complements of the
arrangements in a triple. 

 Let $i\colon M\to M'$ and $j\colon M''\to M'$ be the natural inclusions. Assume that the monodromy of $\LL$ around the distinguished hyperplane $\overline{H}$ is trivial. Then $\LL$ extends to a rank one local system on $M'$, denoted by $\LL'$. Moreover, there is an induced local system  $\LL'' =
j^*\LL'$ on $M''$ by restriction. Call $(\LL,\LL' ,\LL'' )$ a triple of local systems on $(M, M' , M'' ).$
By \cite[Theorem 4]{Cohen02} there is a long exact
sequence in local system cohomology
\begin{equation}\label{les}
   \cdots \to H^k(M',\LL') \to H^k(M,\LL)\to H^{k-1}(M'',\LL'')\to \cdots
\end{equation}

\begin{lemma} \label{lem 2}  Let $\LL$ be a rank one $\K$-local system  on $M$ such that  the monodromy of $\LL$ along a distinguished hyperplane $\barH$ is trivial.
Consider the long exact sequence (\ref{les}). 
If $\dim H^k(M,\LL) =b_k(M) $ for some $0\leq k\leq \ell $, then we have \begin{center}
    $\dim H^k(M',\LL') =b_k(M') $ and $\dim H^{k-1}(M'',\LL'') =b_{k-1}(M'') $.
\end{center}
 %In particular,  and a short exact sequence
%$$0\to H^i(M',\LL') \to H^i(M,\LL)\to H^{i-1}(M'',\LL'')\to  0$$.
\end{lemma}
\begin{proof}
    The classical deletion-restriction method gives the following equality: $$ b_k(M)=b_k(M')+b_{k-1}(M'').$$  The long exact sequence (\ref{les}) implies that 
    $$\dim H^k(M,\LL) \leq \dim H^k(M',\LL')+ \dim H^{k-1}(M'',\LL'').$$
     Then the assumption $\dim H^k(M,\LL) =b_k(M) $ implies that
$$ b_k(M')+b_{k-1}(M'')\leq \dim H^k(M',\LL')+ \dim H^{k-1}(M'',\LL'') .$$
Hence the claims follow from (\ref{inequality}) for $\LL'$ and $\LL''$.  
\end{proof}
\begin{proof}[Proof of Theorem \ref{thm main 1}]
When $k=0$, the claim is obvious. So we assume that $k\geq 1$.  We first reduce the proof to the case $k=\ell$. 
Set $\B=\A \cap F^{k}$, which is an affine arrangement in $F^k$. Since $F^k$ is generic,  
$M(\A)$ is obtained from $M(\B)$ by attaching cells with dimension $\geq k+1$, see \cite{dim-pap}.  
So the  cochain complex for the maximal abelian cover of $M(\B)$ is obtained by truncating the cochain complex (\ref{cochain}) as
$$ 0\to C^{0}(\tilde{X}, \K)\to \cdots  \overset{\partial_{k-1}}{\to}     C^{k}(\tilde{X}, \K)\to 0.$$
It implies that $\dim H^k(M(\A),\LL) \leq \dim H^k(M(\B),\LL|_{M(\B)})$.
%    $$  b_k(M(\A))=\dim H^k(M(\A),\LL) \leq \dim H^k(M(\A\cap F^k),\LL^k) \leq b_k(M(\A\cap F^k)).$$
%    Due to the genericity of $F^k$, we have $ b_k(M(\A))= b_k(M(\A\cap F^k))$, see e.g.  .
%     The claim follows
Then we have $$  b_k(M(\A))=\dim H^k(M(\A),\LL) \leq \dim H^k(M(\B),\LL|_{M(\B)}) \leq b_k(M(\B)).$$ But due to the genericity of $F^k$, we have $ b_k(M(\A))= b_k(M(\B))$. Hence we have $\dim H^k(M(\B),\LL|_{M(\B)}) =b_k(M(\B)) .$ Note that $\LL$ is a constant sheaf if and only if so is $\LL|_{M(\B)}$. Then one can replace $\A$ by $\B$ and $\LL$ by $\LL|_{M(\B)}$, respectively.

  Theorem \ref{thm:localchamber} and Proposition \ref{prop:bchuch} shows that  $\dim H^\ell(M,\LL) =b_\ell(M(\A)) $ if and only if 
\begin{center}
  $\deg(C,C')\cdot \Delta(C, C')=0$ for any $C\in \ch^{\ell-1}(\A)$ and $C'\in \ch^{\ell}(\A)$.  
\end{center}
 Since $\A$ is essential, there always exists a chamber 
$C\in\bch^{\ell-1}(\A)$ such that $\dim X(C)=\ell -1$. 
%Such kind of chamber always exists. 
Note that $C^\lor\in \uch^\ell(\A)$. By Proposition \ref{prop:charsep} and Proposition \ref{prop:degree}, we have $ \deg(C,C^\lor)=1$ and $\Sep(C, C^\lor)=\A$. 
Then $\dim H^\ell(M,\LL) =b_\ell(M(\A)) $ implies that $$0= \Delta(C, C^\lor)=(q_1q_2\cdots q_n)^{\frac{1}{2}}-(q_1q_2\cdots q_n)^{-\frac{1}{2}}$$
Hence $q_\infty^{-1}=q_1q_2\cdots q_n=1.$
Then $\LL$ has trivial monodromy along the hyperplane at infinity $\barH_\infty$. 
Applying Lemma \ref{lem 2}, we have that \begin{center}
    $\dim H^\ell (M',\LL') =b_\ell(M') $ and $\dim H^{\ell-1}(M'',\LL'') =b_{\ell-1}(M'') $.
\end{center}
Note that both $\A'$ and $\A''$ are real hyperplane arrangements. 
If $b_\ell(M')>0$, then $\LL'$ is the constant sheaf  by induction on the dimension of the ambient space and the number of hyperplanes for real hyperplane arrangements. Hence $\LL$ is also the constant sheaf. 
On the other hand, if $b_\ell(M')=0$,  $\A'$ is not essential. It implies that $ \barH_\infty$ is transversal to all other hyperplanes in $\barA$. In particular, $ \LL$ is the constant sheaf if and only if so is $\LL''$. Clearly $\dim H^{\ell-1}(M'',\LL'') =b_{\ell-1}(M'') =b_\ell(M)>0$ in this case. Then the claim follows by induction for $\LL''$.
\end{proof}
\begin{remark} \label{rem nonzero}
    The above proof indeed shows that $\partial_k(q)\neq 0$ for any $0\leq k \leq \ell-1$ if $\A$ is real and $\LL$ is not the constant sheaf.
\end{remark}

\section{Proof of Proposition \ref{prop cup product}} \label{sec:prop}
%In this section, we prove a result for the Aomoto complex similar to  Theorem \ref{thm main 1}.
%With a very similar proof to Theorem \ref{thm main 1}, one can prove the same results for the Aomoto complex. 

%In this section, we always assume that $\A$ is a complex hyperplane arrangement in $\C^\ell$. 

 %The first homology group $H_1(M(\A), \Z)$ is  a free abelian group generated by the meridians $\gamma_1, \dots, \gamma_n$  of hyperplanes. 
%We denote their dual basis by $e_1, \dots, e_n\in H^1(M(\A), \Z)$. 
%Fix a coefficient field $\K$. Consider $w=\sum_{i=1}^n w_i e_i \in H^1(M(\A),\K)$. Note that $w\cup w=0,$ hence  we get the following Aomoto complex by cup product $(H^{\ast}(M(\A),\K),\cup w)$:
%$$\cdots \to  H^{k-1}(M(\A),\K)\overset{\cup w}{\to} H^{k}(M(\A),\K) \overset{\cup w}{\to} H^{k+1}(M(\A),\K)\to\cdots
%$$

%\begin{proposition} Let $\A$ be an essential hyperplane arrangement in $\C^\ell$. 
    %For any non-zero $w\in H^1(M(\A),\K)$,  the cup product map $  H^{k-1}(M(\A),\K)\overset{\cup w}{\to} H^{k}(M(\A),\K)$ is non-zero for any $1\leq k\leq \ell$.
%\end{proposition}
%\begin{proof}
  We give the proof by induction on $\ell$ with two steps.  If $\ell=1$, the claim is obvious. For induction, we assume that the claim holds for any essential hyperplane arrangement in $\C^{\ell-1}$.

\medskip

\noindent Step 1:   Assume that $\A$ is central in $\C^\ell$. We reduce the proof for $\A$ to lower dimension case, then  prove it by induction.

If $\sum_{i=1}^n w_i \neq 0,$ it is well known that the Aomoto complex $(H^{\ast}(M(\A),\K), w\cup)$ is exact, see e.g. \cite[Proposition 2.1]{yuz-bos}. It implies  that  the cup product map $  H^{k-1}(M(\A),\K)\overset{w\cup}{\to} H^{k}(M(\A),\K)$ has rank $ |\sum_{j=0}^{k-1} (-1)^j b_j(M(\A))|$. 
Since $\A$ is essential, it follows from \cite[Lemma 8]{y-generic} that $ |\sum_{j=0}^{k-1} (-1)^j b_j(M(\A))|>0$ for any $1\leq k\leq \ell$.  Then the claim follows. 
On the other hand,  consider the case $\sum_{i=1}^n w_i = 0$. Note that $\A$ is central. Let $U(\A)$ denote the  complement of the
corresponding projective arrangement in $\mathbb{CP}^{\ell -1}$. Then $M(\A)=U(\A)\times \C^*$. One can identify $H^1(U(\A),\K)$ with the affine space $\{x\in \K^n \mid x_1+\cdots+x_n=0\}\cong \K^{n-1}$. Since we assumed that $\sum_{i=1}^n w_i = 0$, $w$ can be viewed as a non-zero element in $H^1(U(\A),\K)$. Moreover, by K\"unneth formula the map  $  H^{k-1}(M(\A),\K)\overset{w\cup}{\to} H^{k}(M(\A),\K)$ is isomorphic to the direct sum of
\begin{center}
    $  H^{k-1}(U(\A),\K)\overset{w\cup}{\to} H^{k}(U(\A),\K) $ and $  H^{k-2}(U(\A),\K)\overset{w\cup}{\to} H^{k-1}(U(\A),\K).$
    \end{center}
 Since $U(\A)$ is a complement 
to an essential hyperplane arrangement with less dimension for its ambient space (in $\C^{\ell-1}$),  the claim follows  by induction. %(Does it mean ``Since $U(\A)$ is a complement to an affine hyperplane arrangement with less number of hyperplanes, the claims is reduced to the case below.''?)

\medskip
 
\noindent Step 2: Next we prove the claim when $\A$ in $\C^\ell$ is not necessarily central by reducing to the central case with the same dimension for the ambient space.

 For $w\neq 0$, without loss of generality, we assume that $w_1 \neq 0$. Since $\A$ is essential, there exists an intersection point $x\in L_0(\A)$ contained in $H_1$. Consider the sub-arrangement $\A_x$ of $\A$ consisting of all hyperplanes passing through $x$. Let  $M(\A_x) $ denote the complement of $\A_x$. Set $M_x =M(\A) \cap B$, where $B$ is a small enough open ball in $\C^\ell$ centered at $x$. 
 We have two inclusions $u\colon M_x \to M(\A) $ and $v\colon M(\A) \to M(\A_x)$. Then 
 $v\circ u\colon M_x \to M(\A_x) $ is a homotopy equivalence.
 %The natural inclusion from $M_x$ to $M(\A)$ induces an surjective map $H^i(M(\A),\Z) \to H^i(M_x,\Z)$.
Hence we have the following commutative diagram by  functoriality 
$$\xymatrix{
     H^{k}(M(\A_x),\K) \ar[r]  &  H^k(M(\A),\K) \ar[r]  &  H^k(M_x,\K)  \\
     H^{k-1}(M(\A_x),\K) \ar[u]^{w_x\cup } \ar[r] &   H^{k-1}(M(\A),\K) \ar[u]^{w\cup} \ar[r] &  H^{k-1}(M_x,\K) \ar[u]^{w_x \cup} 
    },$$
where $w_x=u^*(w)\in  H^1(M_x,\K)\cong  H^1(M(\A_x),\K)$ is non-zero by the assumption $w_1\neq 0$. 
Note that both  horizontal maps on the right square are sujective. Meanwhile, the third vertical map in the above diagram is non-zero for $1\leq k \leq \ell$, since $M_x$ is homotopy equivalent to $M(\A_x)$ and $\A_x$ is a central essential hyperplane arrangement. Putting all together, we get that   $  H^{k-1}(M(\A),\K)\overset{w\cup}{\to} H^{k}(M(\A),\K)$ is non-zero. 
%\end{proof}

\medskip

\noindent
{\bf Acknowledgements.} 
The authors would like to thank Laurentiu Maxim and Botong Wang for helpful discussions on the paper.
Y. Liu is supported by National Key Research and Development Project SQ2020YFA070080, the Project of Stable Support for Youth Team in Basic Research Field, CAS (YSBR-001),  the project ``Analysis and Geometry on Bundles'' of Ministry of Science and Technology of the People's Republic of China and  Fundamental Research Funds for the Central Universities. 
M. Yoshinaga is supported by JSPS KAKENHI 
JP22K18668, JP23H00081.


\begin{thebibliography}{999}
%\pagestyle{plain}

\bibitem{ao-ki}
K. Aomoto, M. Kita, 
Theory of hypergeometric functions. With an appendix by Toshitake Kohno. 
Translated from the Japanese by Kenji Iohara. 
Springer Monographs in Mathematics. Springer-Verlag, Tokyo, 2011. xvi+317 pp

\bibitem{BY} 
P. Bailet, M. Yoshinaga, 
Vanishing results for the Aomoto complex of real hyperplane arrangements via minimality. 
Journal of Singularities, vol. {\bf 14} (2016), 74-90.


%\bibitem{bri-tress}
%E. Brieskorn, 
%Sur les groupes de tresses [d'apr\`es V. I. Arnol'd]. (French) S\'eminaire Bourbaki, 24\`eme ann\'ee (1971/1972), Exp. No. 401, pp. 21-44. Lecture Notes in Math., Vol. 317, Springer, Berlin, 1973. 

\bibitem{Cohen98}
D. Cohen,
Morse inequalities for arrangements.
Adv. Math. {\bf 134} (1998), no. 1, 43-45.

\bibitem{Cohen02}
D. Cohen,  
Triples of arrangements and local systems.
Proc. Amer. Math. Soc.   {\bf 130} (2002), no. 10, 3025-3031.

%\bibitem{cdo}
%D. Cohen, A. Dimca, P. Orlik, 
%Nonresonance conditions for arrangements. 
%\emph{Ann. Inst. Fourier} 
%{\bf 53} (2003), 1883--1896. 

%\bibitem{BT} R. Bott, L. Tu: Differential forms in algebraic topology. Graduate Texts in Mathematics, Springer-Verlag.

\bibitem{dim-pap}
A. Dimca, S. Papadima, 
Hypersurface complements, Milnor fibers and higher homotopy groups of 
arrangements. 
Ann. of Math. (2) {\bf 158} (2003), no. 2, 473--507. 

%\bibitem{esv}
%Esnault, Schechtman, Viehweg, 
%H. Esnault, V. Schechtman, and E. Viehweg, 
%Cohomology of local systems of the complement of hyperplanes. 
%\emph{Invent. Math.} \textbf{109} (1992), 557-561

%\bibitem{fal}
%M. Falk, 
%Arrangements and cohomology. 
%\emph{Ann. Comb.} \textbf{1} (1997), no. 2, 135-157. 

%\bibitem{koh}
%T. Kohno, 
%Homology of a local system on the complement of hyperplanes. 
%\emph{Proc. Japan Acad. Ser. A Math. Sci.} \textbf{62} (1986), no. 4, 144-147. 

%\bibitem{koh-cft}
%T. Kohno, 
%Conformal field theory and topology. 
%Translations of Mathematical Monographs, 210. 
%Iwanami Series in Modern Mathematics. 
%American Mathematical Society, Providence, RI, 2002. x+172 pp. 


%\bibitem{nty}
%S. Nazir, M. Torielli, M. Yoshinaga, 
%On the admissibility of certain local systems. 
%\emph{Topology Appl.} \textbf{178} (2014), 288-299. 

%\bibitem{O-S} 
%P. Orlik, L. Solomon, Combinatorics and topology of complements of hyperplanes. \emph{Invent. Math.} \textbf{56} (1980), 167-189. 

%\bibitem{ot} 
%P. Orlik, H. Terao, 
%Arrangements of hyperplanes. Grundlehren der Mathematischen Wissenschaften, 300. Springer-Verlag, Berlin, 1992. xviii+325 pp.

%\bibitem{ot-int}
%P. Orlik, H. Terao, 
%Arrangements and hypergeometric integrals. Second edition. 
%MSJ Memoirs, 9. Mathematical Society of Japan, Tokyo, 2007. x+112 pp.

%\bibitem{PS}  
%S. Papadima, A. Suciu, 
%The spectral sequence of an equivariant chain complex and homology 
%with local coefficients. 
%\emph{Trans. A. M. S.}, \textbf{362} (2010), no. 5, 2685-2721. 


\bibitem{ps}
S. Papadima, A. I. Suciu, Alexander,
The spectral sequence of an equivariant chain complex and homology with local coefficients.
\emph{Trans. Amer. Math. Soc.} \textbf{362} (2010), no. 5, 2685-2721.



\bibitem{ran}
R. Randell, 
Morse theory, Milnor fibers and minimality of hyperplane arrangements. 
Proc. Amer. Math. Soc. {\bf 130} (2002), no. 9, 2737--2743. 



%\bibitem{stv}
%V. Schechtman, H. Terao, and A. Varchenko, 
%Local systems over complements of hyperplanes and the Kac--Kazhdan conditions for singular vectors. 
%\emph{J. Pure Appl. Algebra} \textbf{100} (1995), 93-102.

%\bibitem{to-yo} 
%M. Torielli, M. Yoshinaga, 
%Resonant bands, Aomoto complex, and real 4-nets.  
%{\em Journal of Singularities}, 
%\textbf{11} (2015) 33-51. 

%\bibitem{var}
%A. N. Varchenko,
%Multidimensional hypergeometric functions the representation theory 
%of Lie Algebras and quantum groups,
%Advanced Series in Mathematical Physics,
%{\bf 2}, World Scientific, 1995.

%\bibitem{wil}
%Williams, (probablye not necessary.) 

\bibitem{y-lef} 
M. Yoshinaga, 
Hyperplane arrangements and Lefschetz's hyperplane section theorem. 
\emph{Kodai Math. J.} \textbf{30}, no. 2 (2007), 157--194.

%\bibitem{y-mil}  
%M. Yoshinaga, 
%Milnor fibers of real line arrangements. 
%\emph{Journal of Singularities}, \textbf{7} (2013), 220-237.


\bibitem{y-generic} 
M. Yoshinaga, 
Generic section of a hyperplane arrangement and twisted Hurewicz maps.
\emph{Topology Appl.} \textbf{155}  (2008), no. 9, 1022--1026.


\bibitem{y-cham} 
M. Yoshinaga, 
The chamber basis of the Orlik-Solomon algebra and Aomoto complex. 
\emph{Ark. Mat.} \textbf{47} (2009), no. 2, 393-407. 


\bibitem{y-mini} 
M. Yoshinaga, 
Minimality of hyperplane arrangements and basis of local system cohomology. 
\emph{Singularities in geometry and topology}, 345-362, 
IRMA Lect. Math. Theor. Phys., 20, \emph{Eur. Math. Soc., Z\"urich}, 2012.

%\bibitem{y-res} 
%M. Yoshinaga, 
%Resonant bands and local system cohomology groups for 
%real line arrangements. 
%\emph{Vietnam J. Math.} \textbf{42} (2014), no. 3, 377-392. 

%\bibitem{yuz}
%S. Yuzvinsky, 
%Orlik-Solomon algebras in algebra and topology. 
%\emph{Russian Math. Surveys} \textbf{56} (2001), no. 2, 293-364.  

\bibitem{yuz-bos}
S. Yuzvinsky, 
Cohomology of the Brieskorn-Orlik-Solomon algebras. 
\emph{Comm. Algebra} \textbf{23} (1995), no. 14, 5339-5354.

\end{thebibliography}
\end{document}